\documentclass[11pt]{article}
\usepackage[utf8]{inputenc}

\usepackage[left=3cm, right=3cm, top=2cm]{geometry}
\usepackage[utf8]{inputenc}
\usepackage{amsmath}
\usepackage{amsthm}
\usepackage{amscd,amssymb}
\usepackage[all]{xy}
\usepackage{color}
\usepackage{mathrsfs}
\usepackage{comment}
\usepackage{appendix}
\usepackage{verbatim}
\usepackage{pdfsync}
\usepackage{hyperref}
\usepackage{cleveref}
\usepackage{float}
\usepackage{setspace}
\usepackage{xpatch}
\usepackage{subcaption}
\usepackage{pst-node}
\usepackage{tikz-cd}
\usetikzlibrary{positioning}
\usetikzlibrary{shapes,arrows}
\usepackage{tikz}
\usepackage{graphicx,epsfig}
\usepackage{mathrsfs}
\usepackage{verbatim}
\usepackage{slashbox}

\date{}

\let\oldproofname=\proofname
\renewcommand{\proofname}{\rm\bf{\oldproofname}}

\newtheorem{thm}{Theorem}[section]
\newtheorem{cor}[thm]{Corollary}
\newtheorem{lem}[thm]{Lemma}
\newtheorem{prop}[thm]{Proposition}

\theoremstyle{definition}

\newtheorem{rmk}[thm]{Remark}
\newtheorem{defn}{Definition}[section]
\newtheorem{claim}[thm]{Claim}

\def\1{\mathbf{1}}

\newcommand{\Z}{\mathbb{Z}}

\newcommand{\N}{\mathcal{N}}
\newcommand{\bN}{\mathbf{N}}

\newcommand{\M}{\mathcal{M}}
\newcommand{\Q}{\mathcal{Q}}
\newcommand{\K}{\mathcal{K}}
\newcommand{\sM}{\mathscr{M}}

\def\D{\text{D}}
\def\S{\text{S}}
\def\conn{\mathrm{Conn}}
\begin{document}
	\title{Vertex cut of  a graph and  connectivity of its neighbourhood complex}
	
	\author{ Rekha Santhanam\footnote{Department of Mathematics, Indian Institute of Technology Bombay, India. reksan@math.iitb.ac.in}, Samir Shukla\footnote{School of Mathematical and Statistical Sciences, Indian Institute of Technology Mandi, India. samir@iitmandi.ac.in}}
	\maketitle
	%\maketitle

	\begin{abstract}
		
	We show that if a graph $G$  satisfies certain conditions, then the connectivity of neighbourhood complex $\mathcal{N}(G)$ is strictly less than the vertex connectivity of $G$. As an application, we give a relation between the connectivity of the neighbourhood complex and the vertex connectivity for stiff chordal graphs, and for weakly triangulated graphs satisfying certain properties. Next, we consider graphs with  a vertex $v$ such that  for any $k$-subset $S$ of neighbours of $v$, there exists a vertex $v_S \neq v$  such that $S$ is  subset of neighbours of $v_S$. We prove that for any graph $G$ with a vertex $v$ as above,  $\mathcal{N}(G-\{v\})$  is $(k-1)$-connected implies that $\mathcal{N}(G)$ is $(k-1)$-connected. 	We  use this to show that:(i) neighbourhood complexes of queen and king graphs are simply connected and (ii) if $G$ is a non-complete stiff chordal graph, then  vertex connectivity of $G$ is $n+1$  if and only if $\conn(\N(G)) = n$.

	\end{abstract}

	\noindent {\bf Keywords} : Neighbourhood complex, vertex connectivity, chordal graphs.
	
	\noindent 2020 {\it Mathematics Subject Classification:} 05C15,  57M15, 05E45  
	
	\vspace{.1in}
	
%	\hrule

	\section{Introduction}
	
	In 1978, L. Lov{\'a}sz (\cite{l}) introduced the notion of a simplicial complex called the neighbourhood complex $\N(G)$ for a graph $G$.  The {\it neighbourhood complex} $\N(G)$ of a graph $G$ is the simplicial complex, whose simplices are those subsets of vertices of $G$ which have a common neighbour. 
	A topological space $X$ is said to be $k$-{\it connected} if every map from an $m$-dimensional sphere $\S^m \to X$ can be extended to
	a map from the $(m+1)$-dimensional disk $\D^{m+1} \to X$ for $m = 0, 1, \ldots , k.$ The connectivity of $X$, denoted  $\conn(X)$, is the largest integer $k$ such that $X$ is  $k$-connected. In \cite{l}, Lov{\'a}sz relates  the chromatic number of a graph $G$ with the connectivity of the neighbourhood complex  $\N(G)$ (see \Cref{lovasz}) and as an application of this  he proved   the  Kneser conjecture, which gives the  chromatic number of a class of graphs called the Kneser graphs. %The graph $G$ is called {\it stiff}, if there is no vertex $u \in V(G)$  which can be folded. 
	
	 In this article, we derive graph theoretic conditions which are sufficient to imply $n$-connectedness of its neighbourhood complex for some appropriate $n \in \mathbb{N}$. 
 In this vein, we first show that if $\N(H)$ is simply connected and the vertex $v \in G=H\cup\{v\}$ satisfies  an appropriate condition  (cf. \Cref{thm:simplyconnected}) then, $\N(G)$ is simply connected.  As an application, we show that the neighbourhood complexes of queen and king graphs are simply connected (cf. \Cref{kingqueen}). This then implies that the connectivity of the neighbourhood complex  can be determined by computing  the homology of the complex (see \Cref{rmk:homologyofqueen}).  
 
% % Given a graph $G$ and  $u,v \in V(G), u \neq v$, if  $N(u) \subseteq N(v)$  then 
%  $G-  \{u\}$ is called a {\it fold} of $G$.  The homotopy type (and hence the connectivity) of the neighbourhood complex of a graph is invariant under the folding operation.  A graph is said to be {\it stiff} if there are no vertices in it which can be folded. 
%  {\blue 	Let $G$ be a graph and $N(u) \subseteq N(v)$  for $u,v \in V(G), u \neq v$.
% The graph $G-  \{u\}$ is called a {\it fold} of $G$ and we denote it by  $G \searrow G -\{u\}$.  We say that $G$ is folded onto $H$, if there exist a sequence of vertices $u_1, \ldots, u_k$ such that $G \searrow G - \{u_1\}, G - \{u_1\} \searrow G-\{u_1, u_2\}, \ldots, G-\{u_1, \ldots, u_{k-1}\} \searrow G- \{u_1, \ldots, u_k\} = H$. The graph $G$ is called {\it stiff}, if there exists no vertex $u \in V(G)$ such that $G \searrow G -\{u\}$.}

 We further prove that for a graph $G$, if there exists a vertex $v$ satisfying the property that  for any $k$-subset $S$ of neighbours of $v$, there exists a vertex $v_S \neq v$  such that $S$ is  subset of neighbours of $v_S$, then  $\N(G-\{v\})$  is $(k-1)$-connected implies that $\N(G)$ is $(k-1)$-connected (cf.  \Cref{thm:general}). As a consequence of this, we show that if $G$ is  an $(n+1)$-connected  chordal graph which cannot be   folded (see \Cref{defn_fold}) onto a clique of size $n+2$, then  $\N(G)$ is  $n$-connected (cf. \Cref{thm:chordal}).

We show that if a graph $G$ satisfies a certain property, then the connectivity of $\N(G)$ is strictly less than that of the vertex connectivity of $G$  (cf. Theorems \ref{thm:cutiscomplete} and \ref{thm:vertexconnectivity}). 	
Finally, in  the case  of chordal graphs, we show that the vertex connectivity completely determines the connectivity of the  neighbourhood complex.
We prove that if $G$ is a non-complete stiff  (see \Cref{defn_fold}) chordal graph, then  vertex connectivity of $G$ is $n+1$  if and only if $\conn(\N(G)) = n$ (cf. \Cref{thm:chordalcomplete}). %In order to prove this,  we prove more general results, which gives  a relation between the vertex connectivity of a graph $G$ and the connectivity of $\N(G)$ for certain classes of graphs. 

In the last section, we show that the inequality between vertex connectivity of the graph and connectivity of the neighbourhood complex varies based  on the class of graphs. Apart from the classes of graphs in Theorems \ref{thm:cutiscomplete} and \ref{thm:vertexconnectivity}),  we give examples of a class of graphs for which vertex connectivity is  higher than the connectivity of its neighbourhood complex. On the other hand, there are classes of graphs for which the inequality is reversed. There are easy examples in the non-stiff case which we describe in \Cref{sec:conclusion}. More interestingly, for any $r \geq 1$ and $p \geq 5$,  we  construct  a stiff  graph $G_{r,p}$ (see \Cref{sec:conclusion}) such that the connectivity of its  neighbourhood complex is $\min\{2r,p-3\}$ but the vertex connectivity is $1$.

	%{\blue	In Theorem \ref{thm:converse}, we  show that for stiff chordal graph $G$, the  vertex connectivity $\kappa(G) > conn(\N(G))$}. {\red write something about the last section}

	\section{Preliminaries}

We begin by defining our objects of interest. These definitions are standard and are available in \cite{bondy} and \cite{dk}. For completeness, we included them here. 	A  graph $G$ is a  pair $(V(G), E(G))$,  where $V(G)$  is called the set of vertices 
	of $G$  and $E(G) \subseteq \binom{V(G)}{2}$
	denotes the set of unordered  edges.
	If $(x, y) \in E(G)$, it is also denoted by $x \sim y$. 
	A {\it subgraph} $H$ of $G$ is a graph with $V(H) \subseteq V(G)$ and $E(H) \subseteq E(G)$.
	For a subset $S \subseteq V(G)$, the induced subgraph $G[S]$ is the subgraph whose set of vertices  $V(G[S]) = S$
	and the set of edges
	$E(G[S]) = \{(v, w) \in E(G) \ | \ v, w \in S\}$. We denote the graph $G[V(G) \setminus S]$ by $G - S$. The complement graph $\bar{G}$ of  $G$ is the graph on $V(G)$  and $E(\bar{G}) = \{(x, y) | \ (x, y) \notin E(G)\} $.
	
	A {\it graph homomorphism} from  $G$ to $H$ is a function
	$\phi: V(G) \to V(H)$ such that, $(v,w) \in E(G) \implies (\phi(v),\phi(w)) \in E(H).$ A graph homomorphism $f$ is called an {\it isomorphism} if $f$ is bijective and $f^{-1}$ is also a graph homomorphism. Two graphs are called {\it isomorphic},  if there exists an isomorphism between them.
	% If $G$ and $H$ are isomorphic, we write $G \cong H$. 
	A {\it clique} of size $n$ or {\it complete graph} on $n$ vertices, denoted by $K_n$,  is a graph on $n$ vertices where any two distinct vertices are adjacent by an edge. If $G$ and $H$ are isomorphic, we write $G \cong H$.
	The {\it chromatic number} $\chi(G)$ of a graph $G$ is defined as
	$\chi(G) := \text{min}\{n \ | \ \exists \ \text{a graph} $ $ \text{ homomorphism from } G \ \text{to} \ K_n\} $.
	
	Let $G$ be a graph and $v$ be a vertex of $G$. The {\it
		neighbourhood  of $v$} is defined as $N_G(v)=\{ w \in V(G) \ |  \
	(v,w) \in E(G)\}$. 
	The {\it degree} of a vertex $v$ is $|N_G(v)|$. For $A \subseteq V(G)$, the neighbour $N_G(A) = \{v \in G \ |  \ v \sim x \ \forall \ x \in A\}$.

	Let $x$ and $y$ be two distinct vertices of $G$. A {\it $xy$-path} $P$ is a sequence $x v_0 \ldots v_n y$ of vertices of $G$ such that $x \sim v_0, v_n \sim y$ and $v_i \sim v_{i+1}$ for all $0 \leq i\leq n-1$. For an $xy$-path $P = x v_0 \ldots v_n y $, the vertex set of $P$ is $V(P) = \{x, v_0, \ldots, v_n, y\}$.  Given an $xy$-path $P = x v_0 \ldots v_n y$, the path $P^{-1} = y v_n \ldots v_0 x$ is an $yx$-path. Given two paths $P= x v_0 \ldots v_m y$ and $Q = y u_0 \ldots u_n z$, we define the $xz$-path $PQ = x v_0 \ldots v_m y u_0 \ldots u_n z$. Two $xy$-paths $P$ and $Q$ are called {\it internally disjoint} if $V(P) \cap V(Q) = \{x, y\}$. A graph $G$ is called {\it $k$-connected} if for any two distinct vertices $x$ and $y$, there exists at least $k$  internally disjoint $xy$-paths. A graph is called connected if it is $1$-connected. The {\it vertex connectivity $\kappa(G)$} is the maximum value of $k$ for which $G$ is $k$-connected.   A {\it vertex  cut} of $G$ is a subset $S \subseteq V(G)$ such that $G-S$ is a disconnected graph. If $X$ is the vertex set of a component of $G-S$, then subgraph $G[S \cup X]$ is called an {\it $S$-component} of $G$. A vertex cut $S$ is called {\it minimal} if $G-S'$ is connected for any  $S'$ such that  $|S'| < |S| $. It is well known  that $\kappa(G) = |S|$, where $S$ is a minimal vertex cut of $G$.

	For $k \geq 3$, a  {\it cycle graph} on $k$ vertices, denoted by $C_k$,  is a graph with $V(C_k) = \{1, \ldots, k\}$ and $E(C_k) = \{(i, i+1) \ | \  1 \leq i \leq k-1\} \cup \{(1, k)\}$. A {\it chordal graph} is a graph having no induced subgraph which is isomorphic to $C_k$ for $k \geq 4$.  It is well known that a chordal graph is a perfect graph, {\it i.e.}, chromatic number of  every  induced subgraph has a clique of that size.

	A {\it finite abstract simplicial complex X} is a collection of
	finite sets such that if $\tau \in X$ and $\sigma \subseteq \tau$,
	then $\sigma \in X$. The elements  of $X$ are called {\it simplices}
	of $X$. The  dimension of a simplex $\sigma$ is equal to $|\sigma| - 1$, here $|\cdot |$ denotes the cardinality. 
	The $0$-dimensional
	simplices are called vertices of $X$ and we denote the set of all vertices of $X$ by $V(X)$. For a subset $S \subseteq V(X)$, the induced subcomplex of $X$ on $S$, denoted $X[S]$, is a simplicial complex  whose simplices are $\sigma \in X$ such that $\sigma \subseteq S$. The $k $th- skeleton of $X$ consists of all simplices of $X$ of dimension at most $k$.

	The \emph{neighbourhood complex} $\N(G)$  of a graph $G$ is the simplicial complex, whose simplices are $\sigma \subseteq V(G)$ such that  $N_G(\sigma) \neq  \emptyset$.  Lov{\'a}sz proved the following statement
		
		\begin{thm} \cite[Lov{\'a}sz]{l} \label{lovasz} For a graph $G$, $\chi(G) \geq \conn(\N(G)) + 3$.
		
	\end{thm}

 \begin{defn}  \label{defn_fold}
 Let $G$ be a graph and $N_G(u) \subseteq N_G(v)$  for $u,v \in V(G), u \neq v$.
The graph $G-  \{u\}$ is called a {\it fold} of $G$ and we denote it by  $G \searrow G -\{u\}$.  % We say that $G$ is folded onto $H$, if there exist a sequence of vertices $u_1, \ldots, u_k$ such that $G \searrow G - \{u_1\}, G - \{u_1\} \searrow G-\{u_1, u_2\}, \ldots, G-\{u_1, \ldots, u_{k-1}\} \searrow G- \{u_1, \ldots, u_k\} = H$. 
The graph $G$ is called {\it stiff}, if there exists no vertex $u \in V(G)$ can be folded. 

\end{defn} 
	
\begin{prop}\cite[Proposition $4.2$ and Proposition $5.1$]{BK} \label{fold} \\
Let $G$ be a graph and $u \in V(G)$. If $G$ is folded onto $G - \{u\}$, then $\N(G)$ is of same homotopy type as $\N(G -\{u\})$.
\end{prop}

In  this article, if no confusion arises, we simply write $n$-connected for both  $n$-vertex connectivity   and  $n$-topological connectivity.  The vertex connectivity is always used in the context of
a graph, and topological connectivity is used in the context of higher dimensional complexes.

	\section{Simply connectedness of neighbourhood complex}\label{sec:main}
	
	In this section, we explore a sufficient condition on a vertex $v$ of a graph $G$ under which  simply connectedness of $\N(G-\{v\})$ implies the simply connectedness of $\N(G)$.  We first recall the following results from \cite{bjorner}, which we use throughout this article.

		\begin{defn} The nerve of a family of sets $(A_i)_{i \in I}$  is the simplicial complex $\mathbf{N} = \mathbf{N}(\{A_i\})$ defined on the vertex set $I$ so that a finite subset $\sigma \subseteq I$ is in $\mathbf{N}$ precisely when $\bigcap\limits_{i \in \sigma} A_i \neq \emptyset$.
	\end{defn}
		\begin{thm}\cite[Theorem 10.6]{bjorner}\label{thm:nerve}
		Let $\Delta$ be a simplicial complex and $(\Delta_i)_{i \in I}$ be a family of subcomplexes such that $\Delta = \bigcup\limits_{i \in I} \Delta_i$.
		\begin{itemize}
			\item[(i)] \label{nerve1}Suppose every non-empty finite intersection $\Delta_{i_1} \cap \ldots \cap \Delta_{i_t}$ for $i_j \in I, t \in \bN$ is contractible, then $\Delta$ and $\mathbf{N}(\{\Delta_i\})$ are homotopy equivalent.
			\item[(ii)] Suppose every  non-empty finite intersection $\Delta_{i_1} \cap \ldots \cap \Delta_{i_t}$ is $(k-t+1)$-connected. Then $\Delta$ is $k$-connected if and only if  $\mathbf{N}(\{\Delta_i\})$ is $k$-connected. 
		\end{itemize}
	\end{thm}
	
	\begin{lem}\cite[Lemma 10.3(ii)]{bjorner}\label{thm:union} Let $\Delta$ be a simplicial complex and $\Delta_1$, $\Delta_2$ be the sub-complexes of $\Delta$ such that $\Delta = \Delta_1 \cup \Delta_2$. If $\Delta_1$  and $\Delta_2$ are $k$-connected and $\Delta_1 \cap \Delta_2$ is $(k-1)$- connected, then $\Delta$ is $k$-connected.
	\end{lem}

Let $X$ be a simplicial complex and let $\eta \in X$.  The {\it link} of $\eta$ is the simplicial complex defined as 
		$$
		lk_X(\eta) := \{\sigma \in X  \ | \ \sigma \cup \eta \in X \ \text{and} \ \sigma \cap \eta = \emptyset\}.
		$$ 
		
		The {\it star} of $\eta$ is defined as 
		$$
		st_X(\eta) := \{\sigma \in X  \ | \ \sigma \cup \eta \in X \}.
		$$

	Observe that  {\it star} of a simplex is always contractible. 
	
	\begin{thm} \label{thm:simplyconnected} Let $G$ be  a connected graph and let there exist a vertex $v$ in G such that for $S=N_G(v)$, the complex   $\N(G-\{v\})[S]$  is path connected. Then $\pi_1(\N(G-\{v\})) = 0$  implies that $\pi_1(\N(G)) = 0$.
	\end{thm}
	\begin{proof}  Suppose $\pi_1(\N(G-\{v\})) = 0$.  Let $S=N_G(v) = \{x_1, \ldots, x_r\}$. For each $1 \leq i \leq r$, let
		$\Delta_i$ be a simplex on vertex set $N_G(x_i) \setminus \{v\}$. Observe that $lk_{\N(G)}(v) = \Delta_1\cup \ldots \cup \Delta_r$. Clearly,  any 
		 non-empty finite intersection $\Delta_{i_1} \cap \Delta_{i_2} \cap \ldots \cap \Delta_{i_t}$, $1 \leq t \leq r$ is a simplex and therefore contractible.   Hence, by Theorem \ref{thm:nerve}$(i)$, $lk_{\N(G)}(v)$ and the nerve $\mathbf{N}(\{\Delta_i\})$ are homotopy equivalent. Since  $\N(G-\{v\})[S]$ is path connected, we observe that $\bN(\{\Delta_i\})$ is path connected and therefore $lk_{\N(G)}(v)$ is path connected. 
		
		Let $\Delta$ be a simplex on $S$. Let $X$ be  the induced subcomplex of  $\N(G)$ on $V(G) \setminus \{v\}$.  Clearly, $X= \N(G-\{v\}) \cup \Delta$ and $\N(G-\{v\}) \cap \Delta = \N(G-\{v\})[S]$.  Since $\N(G-\{v\})$ is simply connected and   $\N(G-\{v\})[S]$ is path connected, using,  \Cref{thm:union} we conclude that $X$ is simply connected. Observe that  $\N(G) = X \cup st_{\N(G)}(\{v\})$ and $X \cap st_{\N(G)} (\{v\})= lk_{\N(G)}(\{v\})$. Since $st_{\N(G)} (\{v\})$  is contractible and $lk_{\N(G)}(\{v\})$ is path connected, the result follows from Lemma \ref{thm:union}. 
		
	\end{proof}

		As an application of, \Cref{thm:simplyconnected} we show that the neighbourhood complexes of queen graphs and king graphs are simply connected.
	
	\begin{defn} The $m\times n$ queen graph $\mathcal{Q}_{m,n}$ is a graph with $mn$ vertices in which each vertex represents a square in an $m\times n$ chessboard, and each edge corresponds to a legal move by a queen (see \Cref{fig:queen and king} (a)).
	\end{defn}
	\begin{defn} The $m\times n$ king graph $\mathcal{K}_{m,n}$ is a graph with $mn$ vertices in which each vertex represents a square in an $m\times n$ chessboard, and each edge corresponds to a legal move by a king (see \Cref{fig:queen and king} (b)).
	\end{defn}

	\begin{figure} [H]
			\begin{subfigure}[]{0.45\textwidth}
				\vspace{1.0 cm}
			\centering
			\begin{tikzpicture}
			[scale=0.35, vertices/.style={draw, fill=black, circle, inner sep=0.5pt, minimum size = 0pt, }]
			
			\node[vertices, label=left:{}] (00) at (0,0) {};
			\node[vertices, label=below:{}] (10) at (4,0) {};
			\node[vertices, label=below:{}] (20) at (8,0) {};
			\node[vertices, label=below:{}] (30) at (12,0) {};
			%	\node[vertices, label=right:{}] (40) at (16,0) {};

			\node[vertices, label=left:{}] (01) at (0,4) {};
			\node[vertices, label=below:{}] (11) at (4,4) {};
			\node[vertices, label=below:{}] (21) at (8,4) {};
			\node[vertices, label=below:{}] (31) at (12,4) {};
			%	\node[vertices, label=right:{}] (41) at (16,4) {};
			
			\node[vertices, label=left:{}] (02) at (0,8) {};
			\node[vertices, label=below:{}] (12) at (4,8) {};
			\node[vertices, label=below:{}] (22) at (8,8) {};
			\node[vertices, label=right:{}] (32) at (12,8) {};

			\node[vertices, label=below:{$\scriptstyle (1,1)$}] (a1) at (2,2) {};
			\node[vertices, label=below:{$\scriptstyle (2,1)$}] (a2) at (6,2) {};
			\node[vertices, label=below:{$\scriptstyle (3,1)$}] (a3) at (10,2) {};
			%	\node[vertices, label=below:{$\scriptstyle (4,1)$}] (a4) at (14,2) {};
			
			\node[vertices, label=left:{$\scriptstyle (1,2)$}] (b1) at (2,6) {};
			\node[vertices, label=above:{$\scriptstyle (2,2)$}] (b2) at (6,6) {};
			\node[vertices, label=right:{$\scriptstyle (3,2)$}] (b3) at (10,6) {};

			\foreach \to/\from in
			{a1/a2, a2/a3, a1/b1, a1/b2, a2/b2, a2/b3, a3/b3,  b1/b2,  b2/b3, a2/b1,  a3/b2} \draw [-] (\to)--(\from);
			\foreach \to/\from in
			{00/10, 10/20, 20/30,  00/01, 01/11, 12/32, 11/21, 21/31,  01/02,10/11, 20/21, 30/31, 02/12, 11/12, 21/22, 31/32} \draw [dashed] (\to)--(\from);

			\path
			
			(a1) edge[bend right= 45] (a3) 
			(b1) edge[bend left= 45] (b3);

			\end{tikzpicture}
					\vspace{0.5 cm}
			\caption{$\Q_{3, 2}$} \label{queen}
		\end{subfigure}
		\begin{subfigure}[]{0.5\textwidth}
			\centering
			\begin{tikzpicture}
			[scale=0.35, vertices/.style={draw, fill=black, circle, inner sep=0.5pt, minimum size = 0pt, }]
			
			\node[vertices, label=left:{}] (00) at (0,0) {};
			\node[vertices, label=below:{}] (10) at (4,0) {};
			\node[vertices, label=below:{}] (20) at (8,0) {};
			\node[vertices, label=below:{}] (30) at (12,0) {};
			\node[vertices, label=right:{}] (40) at (16,0) {};

			\node[vertices, label=left:{}] (01) at (0,4) {};
			\node[vertices, label=below:{}] (11) at (4,4) {};
			\node[vertices, label=below:{}] (21) at (8,4) {};
			\node[vertices, label=below:{}] (31) at (12,4) {};
			\node[vertices, label=right:{}] (41) at (16,4) {};
			
			\node[vertices, label=left:{}] (02) at (0,8) {};
			\node[vertices, label=below:{}] (12) at (4,8) {};
			\node[vertices, label=below:{}] (22) at (8,8) {};
			\node[vertices, label=below:{}] (32) at (12,8) {};
			\node[vertices, label=right:{}] (42) at (16,8) {};
			
			\node[vertices, label=left:{}] (03) at (0,12) {};
			\node[vertices, label=below:{}] (13) at (4,12) {};
			\node[vertices, label=below:{}] (23) at (8,12) {};
			\node[vertices, label=below:{}] (33) at (12,12) {};
			\node[vertices, label=right:{}] (43) at (16,12) {};

			\node[vertices, label=below:{$\scriptstyle (1,1)$}] (a1) at (2,2) {};
			\node[vertices, label=below:{$\scriptstyle (2,1)$}] (a2) at (6,2) {};
			\node[vertices, label=below:{$\scriptstyle (3,1)$}] (a3) at (10,2) {};
			\node[vertices, label=below:{$\scriptstyle (4,1)$}] (a4) at (14,2) {};
			
			\node[vertices, label=left:{$\scriptstyle (1,2)$}] (b1) at (2,6) {};
			\node[vertices, label=below:{$\scriptstyle (2,2)$}] (b2) at (6,6) {};
			\node[vertices, label=below:{$\scriptstyle (3,2)$}] (b3) at (10,6) {};
			\node[vertices, label=right:{$\scriptstyle (4,2)$}] (b4) at (14,6) {};

			\node[vertices, label=above:{$\scriptstyle (1,3)$}] (c1) at (2,10) {};
			\node[vertices, label=above:{$\scriptstyle (2,3)$}] (c2) at (6,10) {};
			\node[vertices, label=above:{$\scriptstyle (3,3)$}] (c3) at (10,10) {};
			\node[vertices, label=above:{$\scriptstyle (4,3)$}] (c4) at (14,10) {};

			\foreach \to/\from in
			{a1/a2, a2/a3, a3/a4, a1/b1, a1/b2, a2/b2, a2/b3, a3/b3, a3/b4, a4/b4, b1/b2, b1/c1, b1/c2, b2/b3, b2/c3, b2/c3, b3/c3, b3/b4, b3/c4, b4/c3, b4/c4, c1/c2, c2/c3, c3/c4, a2/b1, b2/c1, a3/b2, b3/c2, a4/b3, b2/c2} \draw [-] (\to)--(\from);
			\foreach \to/\from in
			{00/10, 10/20, 20/30, 30/40, 00/01, 01/11, 12/32, 11/21, 21/31, 31/41, 40/41, 01/02,02/03, 10/11, 20/21, 30/31, 02/12, 32/42, 11/12, 21/22, 31/32, 41/42, 12/13, 22/23, 32/33, 42/43, 03/13, 13/23, 23/33, 33/43} \draw [dashed] (\to)--(\from);

			\end{tikzpicture}
			\caption{$\K_{4, 3}$} \label{king}
		\end{subfigure}
	\caption{Examples of Queen and King graphs}\label{fig:queen and king}
	\end{figure}
	
	\begin{prop} \label{prop:queen}For any $p, q \geq 2, \N(\Q_{p, q})$ has full $1$-skeleton, {\it i.e.},  $\{(i, j), (k,l)\} \in \N(\Q_{p, q})$ for all 
	$(i,j), (k,l) \in V(\Q_{p,q})$. 
\end{prop}

\begin{proof}
	Let $(i, j), (k, l) \in V(\Q_{p, q})$.  Without loss of generality we assume that $i \leq k$. If $i \neq  k$, then $\{(i, j), (k,l)\} \subseteq N_{\Q_{p, q}}((i, l))$. So, assume that $i=k$. If $q> 2$, then there exists $t \in \{1, \ldots, q\} \setminus \{j, l\}$ and  $\{(i, j), (k,l)\} \subseteq N_{\Q_{p, q}}((i, t))$. If $q= 2$, then without loss of generality we assume that $j=1$ and $l= 2$. In this case, if 
	$(i-1, 1) \in V(\Q_{p, q})$, then $\{(i, j), (k,l)\} \subseteq N_{\Q_{p, q}}((i-1, 1))$ and if  $(i+1, 1) \in V(\Q_{p, q})$, then $\{(i, j), (k,l)\} \subseteq N_{\Q_{p, q}}((i+1, 1))$.
\end{proof}

	\begin{thm} \label{kingqueen} Let $p, q \geq 2$ be positive integers. Then $\N(\Q_{p, q})$ and $\N(\K_{p, q})$ are simply connected.

		\begin{proof}
			We will first induct on $p$ and then on $q$ to get the general result.  Clearly, $\Q_{2, 2} \cong \K_{2,2} \cong K_4$ and therefore 
			$\N(\Q_{2, 2}) \simeq \N(\K_{2,2,}) \simeq \N(K_4) \simeq S^2$, which is simply connected. Let $p, q $ be positive integers where $\min\{p, q\} \geq 2$ and assume that $\N(\Q_{p, q})$ and $\N(\K_{p, q})$ are simply connected. %We first show that $\N(\Q_{p+1, q}$ and  $\N(\K_{p+1, q})$ and  then that $\N(\Q_{p, q+1})$ and  $ \N(\K_{p, q+1})$ are simply connected. 
			We first show that $\Q_{p+1, q}$ is simply connected. 
			%we first prove the following claim.

			For $1 \leq i \leq q$, let $G_i$ be the induced subgraph of $\Q_{p+1, q}$ on vertex set  $V(\Q_{p, q}) \cup \{( p+1, 1), \ldots, ( p+1, i)\}$.  From \Cref{prop:queen}, $\N(\Q_{p, q})$ has full $1$-skeleton and therefore  $\N(\Q_{p,q})[N_{G_1}((p+1, 1))]$ is path connected. Since $\N(\Q_{p, q})$ is simply connected, by using \Cref{thm:simplyconnected}, we conclude that 
			$\N(G_1)$ is simply connected. Fix $ 2 \leq i \leq q$ and inductively assume that $\N(G_{i-1})$ is simply connected.  Write  $N_{G_i} ((p+1, i)) = A \sqcup B$, where, $A = \{(p+1, j) | 1 \leq j \leq i-1\}$ and $B = N_{G_i} ((p+1, i)) \cap V(\Q_{p,q})$. Since $\N(\Q_{p,q})[B]$ is path connected,
			$\N(G_{i-1})[B]$ is also  path connected. Clearly,  $\{(p+1, j), (p, i)\} \subseteq N_{G_{i-1}}((p, j))$ for each $1 \leq j \leq i-1$ and therefore using the fact that  $(p, i) \in B$, we conclude that $\N(G_{i-1})[N_{G_i} ((p+1, i))]$ is path connected. Since $\N(G_{i-1})$ is simply connected, $\N(G_i)$ is simply connected from  \Cref{thm:simplyconnected}.  Hence, by induction 
			$\N(\Q_{p+1, q})$ is simply connected. Similarly, if we are given that $\N(\Q_{p,q})$ is simply connected,  then we can show that 
 $\N(\Q_{p , q+1})$ is simply connected. 
			
			We now show that $\N(\K_{p+1, q})$ is simply connected.  For $1 \leq i \leq q$, let $\K_i$ be the induced subgraph of $\K_{p+1, q}$ on vertex set  $V(\K_{p,q}) \cup \{( p+1, 1), \ldots, ( p+1, i)\}$. Since  $N_{\K_{p+1, 1}} ((p+1, 1)) = \{(p, 1), (p, 2)\} \subseteq N_{\K_{p,q}}((p-1, 1))$, $\K_1$ is simply connected by \Cref{thm:simplyconnected}. Fix $ 2 \leq i \leq q$ and inductively assume that $\N(\K_{i-1})$ is simply connected. If $i < q$, then $N_{\K_i}((p+1, i)) = \{(p+1, i-1),(p, i-1), (p, i), (p, i+1)\}$  and if $i = q$, then $N_{\K_i}((p+1, i)) = \{(p+1, i-1),(p, i-1), (p, i)\}$.  Here, since 
			$\{(p+1, i-1), (p, i-1)\} \subseteq N_{\K_{i-1}}((p, i)), \{(p, i-1), (p, i)\} \subseteq N_{\K_{i-1}}((p-1, i)) $ and $\{(p, i), (p, i+1)\} \subseteq N_{\K_{i-1}}((p-1, i))$, we have that $\N(\K_{i-1})[N_{\K_i}((p+1, i))]$ is path connected. Since $\K_{i-1}$ is simply connected, $\K_i$ is simply connected  by \Cref{thm:simplyconnected}. From  induction, $\K_{p+1, q}$ is simply connected. By similar argument, we can show that $\N(K_{p,q})$ is simply connected implies 
			$\N(\K_{p, q+1})$ is simply connected.  Hence, the theorem is proved.
		   
		\end{proof}
		
	\end{thm}

\begin{rmk} \label{rmk:homologyofqueen}
In \Cref{table:1}, we have computed the homology of neighbourhood complexes of queen graphs for some values of $m$ and $n$ by computer (using SAGE). Since $Q_{m, n} \cong Q_{n, m}$ and  neighbourhood complexes of queen graphs are simply connected, we conclude that $\conn(\N(Q_{m, n})) = 2$ for $2 \leq m \leq 4, 5\leq n \leq 6$. 
%\ssnote{Can we write a conjecture for the $\conn(\N(Q_{m, n}))$ for any $m, n$ ?}\rsnote{yes we should but maybe should discuss of possible approaches to this.} 
\end{rmk}

\begin{table}[H]
	\centering
%	\footnotesize
\scalebox{0.58}{
	\begin{tabular}{|c|c|c|c|c|c|c|c|c|c|c|c|c|c|c|c|c|c|c|c|c|c|}
		\hline 
	\backslashbox[20mm]{k}{(m,n)}&$(2,2)$& $(2,3)$& $(2,4)$& $(2,5)$& $(2,6)$& $(2,7)$& $(2,8)$&$(2,9)$& $(2,10)$&$(3,3)$ & $(3,4)$& $(3,5)$& $(3,6)$ & $(3,7)$&$(3,8)$& $(4,2)$& $(4,4)$& $(4,5)$& $(4,6)$\\ [10pt]
	\hline
	\hline
	$0$&$0$& $0$& $0$& $0$& $0$& $0$& $0$&$0$& $0$&$0$ & $0$& $0$& $0$ & $0$&$0$& $0$&  $0$& $0$& $0$\\ [5pt]
	\hline
		$1$&$0$& $0$& $0$& $0$& $0$& $0$& $0$&$0$& $0$&$0$ & $0$& $0$& $0$ & $0$&$0$& $0$&  $0$& $0$& $0$\\ [5pt]
		\hline
			$2$&$\Z$& $\Z$& $\Z$& $0$& $0$& $0$& $0$&$0$& $0$ &$0$ & $0$& $0$& $0$ & $0$&$0$& $\Z$&  $0$& $0$& $0$\\ [5pt]
		\hline
			$3$&$0$& $0$& $0$& $\Z^3$& $\Z$& $\Z$& $\Z$&$\Z$& $\Z$ &$\Z^3$ & $\Z^5$& $\Z^{11}$& $\Z^{8}$ & $\Z^5$&$\Z^3$& $0$&  $\Z^5$& $\Z^9$& $\Z^4$\\ [5pt]
		\hline
\end{tabular}}
\caption{\small{Reduced Homology  $\tilde{H}_k(\N(Q_{m,n});\Z)$}}
\label{table:2}
\end{table}

%\begin{table}[H]
%	\centering
%	\footnotesize
%\scalebox{0.58}{
%	\begin{tabular}{|c|c|c|c|c|c|c|c|c|c|c|c|c|c|c|c|c|c|c|c|c|c|}
	%	\hline 
%	(m, n)&$(2,2)$& $(2,3)$& $(2,4)$& $(2,5)$& $(2,6)$& $(2,7)$& $(2,8)$&$(2,9)$& $(2,10)$&$(3,3)$ & $(3,4)$& $(3,5)$& $(3,6)$ & $(3,7)$&$(3,8)$& $(4,2)$& $(4,4)$& $(4,5)$& $(4,6)$\\ [10pt]
%	\hline
%	$\kappa(Q_{m, n})$&$3$& $4$& $5$& $6$& $7$& $8$& $9$&$10$& $11$&$6$ & $7$& $8$& $9$ & $10$&$11$& $5$&  $9$& $10$& $11$\\ [5pt]
%	\hline
%\end{tabular}}
%\caption{\small{Vertex connectivity of $Q_{m, n}$}}
%\label{table:1}
%\end{table}

	\section{Vertex cut and connectivity of neighbourhood complex}

	In this section, we consider graphs $G$ which can be written as a union of two subgraphs $G_1$, $G_2$ such that the connectivity of their neighbourhood complexes are known.  We then  give conditions on the the subgraph induced by their intersection  to give an estimate  for  the  connectivity of the neighbourhood complex of $G$.

	 We first recall the following results from \cite{bjorner}.
	 
		\begin{lem}\cite[Lemma 10.3(iii)]{bjorner}\label{thm:intersection}Let $\Delta_1$ and $\Delta_2$ be two simplicial complexes. If $\Delta_1 \cap \Delta_2$ and $\Delta_1 \cup \Delta_2$ are $k$-connected, then so are $\Delta_1$ and $\Delta_2$.
	\end{lem}

	\begin{lem}\cite[Lemma $10.4(ii)$]{bjorner}\label{lem:suspension}
		Let $\Delta_1$ and $\Delta_2$ be two contractible subcomplexes of a simplicial complex $\Delta$ such that $\Delta=\Delta_1 \cup \Delta_2$. Then $\Delta \simeq \ \Sigma(\Delta_1 \cap \Delta_2)$, where $\Sigma(X)$ denotes the suspension of space $X$.
	\end{lem}
	
	For a positive integer $n$, let $[n]$ denote the set $\{1, \ldots, n\}$. We first consider the case when a graph can be written as a union of two subgraphs whose intersection is a complete graph which is connected to a vertex in each subgraph. 
	\begin{thm} \label{thm:cutiscomplete} Let $G = G_1 \cup G_2$ such that $G_1 \cap G_2 =  H \cong K_{n}$. Let there exist  $a \in V(G_1), b \in V(G_2)$ such that $a$ and $b$ are adjacent to each vertex of $H$. If $\N(G_1)$ and $\N(G_2)$ are $(n-1)$-connected, then $\conn(\N(G)) = n-1$.
	\end{thm}
%	\rsnote{I was looking at Csorba's paper. Even though statements are different. Its best to make sure he doesn't mention any thing like what we are doing }
	
	\begin{proof}
		
		Let $V(H) = \{x_1, \ldots, x_n\}$. For each $1 \leq i \leq n$, let $\Delta_i$ be a simplex on the vertex set  $N_G(x_i)$. Let $\Delta_{n+1} = \N(G_1)$ and $\Delta_{n+2} = \N(G_2)$.  Then $\N(G) = \Delta_1 \cup \ldots \cup \Delta_{n+2}$. We compute the nerve $\mathbf{N}(\{\Delta_{i, 1 \leq i \leq n+2}\})$. First we show that, for any subset, $\{i_1, \ldots, i_{n+1}\} \subset [n+2]$ the  intersection $\Delta_{i_1} \cap \ldots \cap \Delta_{i_{n+1}}$ is  non-empty. Let $j = [n+2] \setminus \{i_1, \ldots, i_{n+1}\}$.  If $j = n+1 $, then $\{b\} \in \Delta_{i_1} \cap \ldots \cap \Delta_{i_{n+1}} $ and if $j= n+2$, then $\{a\} \in \Delta_{i_1} \cap \ldots \cap \Delta_{i_{n+1}} $. If $j \in [n]$, then $\{x_j\} \in \Delta_{i_1} \cap \ldots \cap \Delta_{i_{n+1}}  $. Since $\bigcap\limits_{i=1}^{n+2} \Delta_i = \emptyset$, we conclude that  $\mathbf{N}(\{\Delta_{i, 1\leq i \leq n+2}\})$ homotopy equivalent to the simplicial boundary of an $(n+1)$-dimensional simplex, which is of the same homotopy type as $S^{n}$.

		Observe that  $\Delta_{n+1} \cap \Delta_{n+2}$ is a simplex on vertex set  $\{x_1, \ldots, x_n\}$. Further, $\Delta_{n+1} \cap \Delta_{j_1} \cap  \ldots \cap \Delta_{j_k}$,  $\Delta_{n+2} \cap \Delta_{j_1} \cap  \ldots \cap \Delta_{j_k}$ and  $ \Delta_{j_1} \cap  \ldots \cap \Delta_{j_k}$ are all simplices of $\N(G)$ for all $\{j_1, \ldots, j_k\} \subset [n]$. Therefore, each non-empty intersection  $\Delta_{i_1} \cap \ldots \cap \Delta_{i_t}$ is a simplex of $\N(G)$ for $2 \leq t \leq n+2$. Since $\N(G_1)$ and $\N(G_2)$ are $(n-1)$-connected, by taking $k = n-1$ in Theorem \ref{thm:nerve}$(ii)$ we conclude that $\conn(\N(G)) \geq n-1$.
		
		Suppose $\N(G)$ is $n$-connected. 	Let $X = \N(G_2) \cup \bigcup\limits_{i=1}^{n} \Delta_i$. 
		
		\begin{claim} \label{claim:X1X2} $\N(G_1) \cap X \simeq S^{n-1}$.
		\end{claim}

		\begin{proof}[Proof of Claim \ref{claim:X1X2}]
			
			For $1 \leq i \leq n$, let $\Gamma_i = \Delta_i \cap \N(G_1)$ and let $\Gamma_{n+1}$ be a simplex on vertex set  $\{x_1, \ldots, x_n\}$. Then $\N(G_1) \cap X = \bigcup\limits_{i=1}^{n+1}\Gamma_i$.
			Since each $\Gamma_i$ is a simplex, we see that each non-empty intersection $\Gamma_{i_1} \cap \ldots \cap \Gamma_{i_t}$ is a simplex and therefore contractible. Thus, from  Theorem \ref{thm:nerve}$(i)$, $\N(G_1) \cap X \simeq \mathbf{N}(\{\Gamma_{i, 1 \leq i \leq n+1}\})$. Observe that $\bigcap\limits_{i=1}^{n+1} \Gamma_i = \emptyset$. Further, for any $1 \leq t \leq n$,  $\{a\} \in \Gamma_{i_1} \cap \ldots \cap \Gamma_{i_t}$ if $n+1 \notin \{i_1, \ldots, i_t\}$ and $\{x_j\} \in  \Gamma_{i_1} \cap \ldots \cap \Gamma_{i_t}$ if $n+1 \in \{i_1, \ldots i_t\}$, where $n+1 \neq j \in [n] \setminus \{i_1, i_2 \ldots, i_t\}$. Therefore, $\mathbf{N}(\{\Gamma_{i, 1 \leq i \leq n+1}\}) \simeq S^{n-1}$.
		\end{proof}

		Clearly $\N(G) = \N(G_1) \cup X$. 	By Mayer-Vietoris sequence for homology,  
		\begin{center}
			$ \cdots \longrightarrow \tilde{H}_n(\N(G))\longrightarrow \tilde{H}_{n-1}(\N(G_1) \cap X) \longrightarrow \tilde{H}_{n-1}(\N(G_1)) \oplus \tilde{H}_{n-1}(X) \longrightarrow \tilde{H}_{n-1}(\N(G)) \longrightarrow \cdots$.
		\end{center}

		Since $\N(G)$ is $n$-connected and $\N(G_1)$ is $(n-1)$-connected, we have  $\tilde{H}_{n}(\N(G)) = 0,$  $\tilde{H}_{n-1}(\N(G)) = 0$ and $\tilde{H}_{n-1}(\N(G_1)) = 0$. So, $\tilde{H}_{n-1}(X) \cong \tilde{H}_{n-1} (\N(G_1) \cap X) $ and therefore, Claim \ref{claim:X1X2} implies that $\tilde{H}_{n-1}(X) \cong \mathbb{Z}$.

			Let $\Delta = \bigcup\limits_{i=1}^n \Delta_i$. Then $X = \N(G_2) \cup \Delta$. Since $a, b \in \bigcap\limits_{i=1}^n \Delta_i$, by using Theorem \ref{thm:nerve}$(i)$, we conclude that $\Delta$ is contractible. We now show that $\N(G_2)\cap \Delta$ is contractible. For each $1 \leq i \leq n$, let $T_i = \Delta_i \cap \N(G_2)$. Then $\N(G_2)\cap \Delta = T_1 \cup \ldots \cup T_n$. Since each $T_i$ is a simplex and $b \in \bigcap\limits_{i=1}^{n} T_i$, we see that  $\mathbf{N}(\{T_{i, 1 \leq i \leq n}\}) \simeq \N(G_2)\cap \Delta$ and $\mathbf{N}(\{T_{i, 1 \leq i \leq n}\}) $ is an $(n-1)$-dimensional simplex. Therefore, $\N(G_2)\cap \Delta$ is contractible. By Mayer-Vietoris sequence for homology, we have 
			\begin{center}
				$ \cdots \longrightarrow \tilde{H}_{n-1}(\N(G_2))  \oplus \tilde{H}_{n-1}(\Delta) \longrightarrow \tilde{H}_{n-1}(X) \longrightarrow \tilde{H}_{n-2}(\N(G_2) \cap \Delta) \longrightarrow \cdots$.
			\end{center}
			
			Since $\N(G_2)$ is $(n-1)$-connected, $\tilde{H}_{n-1}(\N(G_2)) = 0$. Further, since $\Delta $ and $\N(G_2) \cap \Delta$ is contractible, we conclude that $\tilde{H}_{n-1}(X) = 0$, which is a contradiction.
	\end{proof}

	We now prove that even if the intersection of the two subgraphs of $G$ is not a complete graph, we get a result about the cardinality of the intersection in terms of the connectivity of $\N(G)$.
	\begin{thm}\label{thm:vertexconnectivity} Let $G = G_1 \cup G_2$ and $V(G_1) \cap V(G_2) = S$. Let there exist $a \in V(G_1), b \in V(G_2)$ such that $a \sim x, b \sim x$ for all $x \in S$. Let $k = \min\{\conn(\N(G_1)), \conn( \N(G_2))\}$. 
	\begin{itemize}
	    \item[($i$)]  	If $k \geq |S|$, then  $	|S| \geq 	\conn(\N(G)) +1.$
	    \item[($ii$)] $\conn(\N(G)) \leq k$. In particular, if $k \leq \conn(\N(G[S]))$, then $|S| \geq 	\conn(\N(G)) +3.$
	\end{itemize}
	\end{thm}

	\begin{proof}
	If $S = \emptyset$, then $G$ is a disconnected graph and therefore $\N(G)$	is disconnected. Since connectivity of a disconnected space is $-1$, result is true in this case. So assume that $S \neq \emptyset$.  Let $S = \{x_1, \ldots, x_n\}$.	For each $1 \leq i \leq n$, let $\Delta_{x_i}$ be a simplex on vertex set $N_G(x_i)$, $\Delta_{n+1} = \N(G_1)$ and $\Delta_{n+2} = \N(G_2)$. Clearly, $\N(G) = \Delta_{n+1} \cup \Delta_{n+2} \cup \bigcup\limits_{i=1}^{n} \Delta_{x_i}$. Let  $\mathbf{N} := \mathbf{N}(\{\Delta_1, \Delta_2, \Delta_{x_i,  1 \leq i \leq n} \} )$.  For a simplicial complex $\mathcal{K}$, let $\M(\mathcal{K})$ denotes the set of maximal simplices of $\mathcal{K}$.  We first prove the following:
		\begin{claim} \label{claim:intersection}  $\mathbf{N} \simeq \Sigma (\Sigma (\N(G[S])))$.
		\end{claim}
		
		\begin{proof}[Proof of Claim \ref{claim:intersection}] 
			Observe that  $\M(\mathbf{N} )= \{\sigma \cup \{n+1, n+2\} \ | \ \sigma \in \M(\N(G[S]))\} \cup \{S  \cup \{n+1\} , S  \cup \{n+2\} \}$.  Write    $\mathbf{N} = X_1 \cup X_2$, where, $\M(X_1) = \{\sigma \cup \{n+1, n+2\} \ | \ \sigma \in \M(\N(G[S]))\} $ and $\M(X_2) = \{S \cup \{n+1\}, S \cup \{n+2\}\}$.  Observe that, $\M(X_1 \cap X_2) = \{\sigma \cup \{n+
			1\} \ | \ \sigma \in  \M(\N(G[S])) \} \cup \{\sigma \cup \{n+2\} \ | \ \sigma \in \M(\N(G[S])) \}$ and thus we conclude that $ X_1 \cap X_2 \simeq \Sigma(\N(G[S]))$. Since $X_1$ and $X_2$ are contractible, the result follows from Lemma \ref{lem:suspension}. This completes the proof of Claim \ref{claim:intersection}.
		\end{proof}
		
		\begin{itemize}
		    \item[($i$)] 	Let $k \geq n$. Since $\chi(G[S]) \leq n$, by \Cref{lovasz} we see that $\conn(\N(G[S])) \leq n-3$. Hence,  $\conn(\mathbf{N}) \leq n-1$ by Claim \ref{claim:intersection}. Observe that each non-empty intersection $\Delta_{i_1} \cap \ldots \cap \Delta_{i_t}, t \geq 2$ is a simplex for $\{i_1, \ldots, i_t\} \subseteq S \cup \{n+1,n+ 2\}$. Also, since $\Delta_{n+1}, \Delta_{n+2}$ are $n$-connected and $\Delta_{x_i}$ is a simplex for all $1 \leq i \leq n$, we see that each non-empty intersection $\Delta_{i_1} \cap \ldots \cap \Delta_{i_t}$ is $(n-t+1)$-connected for all $t \geq 1$ and $i_1, \ldots, i_t \in S \cup \{n+1,n+2\}$. Since $\conn(\mathbf{N}) \leq n-1$, using 
		Theorem \ref{thm:nerve}$(ii)$ we conclude that $\conn(\N(G)) \leq n-1$. 
		\item[($ii$)] 	Without loss of generality, we assume that $k = \conn(\N(G_2))$. 	Let $Y_1 = \N(G_1) \cup \bigcup\limits_{i=1}^n \Delta_{x_i}$. Clearly, $\N(G) = Y_1 \cup \N(G_2)$.  Observe that $\M(Y_1 \cap \N(G_2))= \{ N_{G_2}(x) \ | \ x \in S\}  \cup \{S\}$.  For each $1 \leq i \leq n$,  let $Z_{x_i}$ be the  simplex on vertex set  $N_{G_2}(x_i)$ and $Z_{n+1}$ be the simplex on vertex set  $S$. Clearly, $Y_1 \cap \N(G_2) = Z_{n+1} \cup \bigcup\limits_{i=1}^{n} Z_{x_i}$.
	
		Observe that $\M(\mathbf{N}(\{Z_{n+1}, Z_{x_i, 1 \leq i \leq n}\}) ) = \{S \} \cup \{\sigma \cup \{n+1\} \ | \ \sigma \in \M(\N(G[S]))\}$. Since $V(\N(G[S])) \subseteq S$, we see that $ \mathbf{N}(\{Z_{n+1}, Z_{x_i, 1 \leq i \leq n}\}) \simeq \Sigma (\N(G[S]))$. Since each $Z_{x_i}$  and $Z_{n+1}$ are simplices, any  non-empty finite intersection of these simplices is also a simplex and hence contractible. 
		Thus $Y_1 \cap \N(G_2) \simeq \mathbf{N}(\{Z_{n+1}, Z_{x_i, 1 \leq i \leq n}\}) \simeq \Sigma(\N(G[S]))$. Therefore, $Y_1\cap \N(G_2)$ is at least $(k+1)$-connected. If $\N(G)$ is $(k+1)$ connected, then from Lemma  \ref{thm:intersection}, $\N(G_2)$ has to be $(k+1)$-connected, which is a contradiction. Hence, $\conn(\N(G)) \leq k $. If 
		$k \leq \conn(\N(G[S])) $, then we have $\conn(\N(G)) \leq k \leq \conn(\N(G[S])) \leq n-3$.
	
		\end{itemize}
	\end{proof}

	 The {\it complement graph} $\bar{G}$ of  $G$ is the graph with vertex set same as of  $V(G)$  and $E(\bar{G}) = \{(x, y) | \ (x, y) \notin E(G)\} $.	A graph $G$ is called \emph{weakly triangulated}, if it has no induced subgraph isomorphic to  a cycle with five or more vertices, or to the complement of such a cycle.

	\begin{thm}\cite[Theorem 1] {hayward} \label{thm:hayward} Let $S$ be a minimal vertex cut of a weakly triangulated graph $G$ and let $S$ induces a connected subgraph of $\bar{G}$. Then each component of $G - S$ includes at least one vertex adjacent to all the vertices of $S$.
		
	\end{thm}
	
	The above property of weakly triangulated graphs and Theorem \ref{thm:vertexconnectivity} imply  the following statement. 
	\begin{thm} \label{thm:weaklytriangulated} Let $G$ be  a weakly triangulated graph and $S$ be a minimal vertex cut of $G$ which induces a connected subgraph of $\bar{G}$. Let $G_1, G_2$  be  subgraphs of $G$ such that 
		 $G = G_1 \cup G_2$ and    $V(G_1 \cap G_2) = S$. If for  $k = \min\{\conn(\N(G_1)), \conn( \N(G_2))\}$,  either $k \geq |S|$ or  $k \leq \conn(\N(G[S]))$, then 
		$|S| \geq 	conn(\N(G)) +1$.
	
	\end{thm}

	\section{Vertex connectivity and neighbourhood complexes of chordal graphs}

	In this section, we show that  for stiff chordal graphs, the vertex connectivity completely determines the connectivity of its neighbourhood complex. The following theorem is a generalization of  \Cref{thm:simplyconnected}.
	
	\begin{thm}\label{thm:general}
		
		Let $G$ be a graph and $n \geq 0$ be an integer. Let $v\in V(G)$  such that for each $S \subseteq N(v), |S| \leq n+1$ there exists a vertex $v_S \neq v$ satisfying $S \subseteq N_G(v_S)$.  Then for all   $k \leq n$,  $\N(G - \{v\})$ is $k$-connected implies  $\N(G)$ is $k$-connected. 
	\end{thm}

	\begin{proof}
	Let   $N_G(v) = \{x_1, \ldots, x_r\}$.	Assume that  $\N(G-\{v\})$ is $k$-connected.

		\begin{claim}  \label{claim:link}
			%			$lk(v) \simeq \mathbb{S}^{|N(v)|-2}$.
			The simplicial complex 	$lk_{\N(G)}(v)$ is at least $(n-1)$-connected.
		\end{claim}
		\begin{proof}[Proof of Claim \ref{claim:link}]
			For each $1 \leq i \leq r$, let  $\Delta_i$ be a simplex on vertex set $N_G(x_i) \setminus \{v\}$. Then observe that $lk_{\N(G)}(v) = \Delta_1\cup \ldots \cup \Delta_r$. Clearly,  any non-empty finite intersection $\Delta_{i_1} \cap \Delta_{i_2} \cap \ldots \cap \Delta_{i_t}$, $1 \leq t \leq r$ is a simplex and therefore contractible. Hence, by Theorem \ref{thm:nerve}$(i)$, $lk_{\N(G)}(v)$ and the nerve $\mathbf{N}(\{\Delta_i\})$ are homotopy equivalent. Let $S= \{i_1, \ldots, i_t\} \subseteq  \{1, \ldots, r\}$ and $t\leq n+1$. Then by assumption there exists a vertex $v_S \neq v$ such that $\{x_{i_1}, \ldots, x_{i_t}\} \subseteq N_G(v_S)$ and thereby showing that $v_S \in \bigcap\limits_{i \in S}\Delta_{i}$. Hence $S \in \bN(\{\Delta_i\})$. Thus $\bN(\{\Delta_i\})$ has full $n$-skeleton and therefore result follows.
			This completes the proof of Claim \ref{claim:link}.
		\end{proof}

		Let $X = \N(G - \{v\})$ and $Y = \N(G)[V(G) \setminus \{v\}]$.  We now show that $Y$ is at least $k$-connected.
		%	\begin{claim}  \label{claim1}
		%		$Y$ is at least $k$-connected.	
		%	\end{claim}
		Let $\Delta$ be a simplex on vertex set $N_G(v)$. Observe that $X \cup \Delta= Y$ and $V(X \cap \Delta) = N_G(v)$.  Let $\sigma \subseteq N_G(v)$ such that  $|\sigma| \leq n+1$. By assumption there exists a vertex $v_{\sigma}\neq v$ such that $\sigma \subseteq N_G(v_{\sigma})$ and thereby showing that $\sigma \in X \cap \Delta$. Hence, $X \cap \Delta$  has a full $n$-skeleton, and therefore it is at least $(n-1)$-connected. Since $k \leq n$, $X \cap \Delta$ is at least $(k-1)$-connected. Further, since $X$ is $k$-connected  and $\Delta$ is contractible, using Lemma \ref{thm:union} we conclude that $Y$ is $k$-connected. %\rsnote{Did we mention what we mean by full skeleton in the preliminaries? If not please add it there } 
		
		Observe that  $\N(G) = Y \cup st_{\N(G)}(v)$ and $Y \cap st_{\N(G)} = lk_{\N(G)}(v)$.
		Since $st_{\N(G)}(v)$ is a cone over $v$, it is contractible. Further, since $Y$ is $k$-connected, \Cref{thm:general} follows from Claim \ref{claim:link} and Lemma \ref{thm:union}. 
	\end{proof}

		As an application of Theorem \ref{thm:general}, we prove that under mild assumptions the neighbourhood complex  of $(n+1)$-connected chordal graph is $n$-connected.  To prove this, we first  establish  two lemmas for chordal graphs.
%	\section{Application of Theorem \ref{thm:general}}
We recall the following result from \cite{bondy}.
	
		\begin{thm}\cite[Theorem 9.21]{bondy} \label{thm:simplicial} Every chordal graph which is not complete has two non-adjacent simplicial\footnote{ A vertex $v$ is called  {\it simplicial} if  $G[N_G(v)]$ is a clique.} vertices.
	\end{thm}

	\begin{lem} \label{lem:simplicial}Let $n \geq 1$ and let $G$ be a $n$-connected non-complete chordal graph. Then there exists a  simplicial vertex $v$ such that $\chi(G- \{v\}) = \chi(G)$ and $G - \{v\}$ is  $n$-connected.  
	\end{lem}
	
	\begin{proof}
		Since chordal graphs are perfect graphs and $G$ is a non-complete,   Theorem \ref{thm:simplicial} implies that there exists a simplicial vertex $v$ such that $\chi(G -  \{v\}) = \chi(G)$. We show that $G - \{v\}$ is $n$-connected. To prove this, it is enough to show that for any two vertices $x, y \in V(G- \{v\})$ there exist $n$ internally disjoint $xy$-paths  in $G - \{v\}$.
		Let $x, y \in V(G - \{v\})$. Since $G$ is $n$-connected, we get $n$ internally disjoint $xy$-paths $P_1, \ldots, P_n$ in $G$.  Without loss of generality, we  assume that all these paths are the shortest disjoint paths. If $v \notin V(P_i)$ for all $1 \leq i \leq n$, then all $P_i^{'}s$ are $xy$-paths in $G-\{v\}$ and we are done. 
		
		So, assume that there  exists  $1 \leq i \leq n$ such that $v \in V(P_i)$. Without loss of generality, we  assume that $i=1$. We  will replace the path $P_1$ by a $xy$-path $P_1'$ in $G-\{v\}$ such that $P_1'$ is internally disjoint from $P_i$ for all $2 \leq i \leq n$.

		If $x, y \notin N_G(v)$, then there exist $w, w' \in V(P_1)$ distinct from $x$ and $y$ such that $w \sim v  \sim w'$. Since $v$ is a simplicial vertex,  $w \sim w'$ and we get  a $xy$-path $P_1^{'} = x \ldots w  w' \ldots y$ from $P_1 = x \ldots w vw' \ldots y$ by removing $v$ from $P_1$, which contradict the fact that $P_1$  is  the shortest path. Hence, $\{x, y\} \cap N_G(v) \neq \emptyset$. Now, suppose  $|\{x, y\} \cap N_G(v)| = 1$ and say  
		$x \sim v$.  Let   $P_1 = x \ldots v w_1 \ldots w_k y, k \geq 1$.  In this case we  can replace the path $P_1$ by  $P_1' = x w_1 \ldots w_k y$, which is again a contradiction. Hence,  $x, y \in N_G(v)$. Let $N_G(v) = \{x, y, z_1, \ldots, z_m\}$.  If there exists a $z \in N_G(v)$ different from $x, y$ such that $z \notin V(P_i)$ for all $1 \leq i \leq n$, then we replace the path $P_1$ by the path $P_1'= x z y$.  Since $G$ is $n$-connected, $m \geq n-2$. If $m \geq n-1$, then we can easily construct $m+1$ internally disjoint $xy$-paths, namely $P_1 = xy, P_2 = xz_1 y, \ldots, P_{m+1} = x z_m y$.  
		
		So, assume that $m = n-2$, {\it i.e.}, $deg(v) = n$  and for each $z \in N_G(v)$ there exists an $i$ such that $z \in V(P_i)$.	Since $P_1, \ldots, P_n$ are the shortest paths, we can assume that $P_1 = x v y, P_2 = xy, P_3 = xz_1 y, \ldots, P_n = x z_{n-2} y$.
		
		Since $G$ in non-complete, there exists $w \in V(G)$ such that $w \notin \{v\} \cup N_G(v)$. Further,  since $G$ is $n$-connected, we have $n$ internally  disjoint $wx$ paths $L_1, \ldots, L_n$  and $n$ internally disjoint $wy$ paths $Q_1, \ldots, Q_n$  in $G$. 
		We consider the following cases.
		
		{\bf Case 1.} $v$ does not belong to $V(L_i)$ or $V(Q_j)$ for all $1 \leq i, j \leq n$.
		
		If $x \sim w$ and $y \sim w$, then we replace $P_1$ by $xwy$. If $x \not\sim w$ and $y \sim w$, then 	since $deg(v) = n$, there exists $j_1$ such that $V(L_{j_1})$  is disjoint from $y, z_1, \ldots, z_{n-2}$. %Observe that since $N(v) = \{x, y , z_1, \ldots, z_{n-2}\}$ and $V(L_{j_1}) \cap \{y, z_1, \ldots, z_{n-2} \} = \emptyset, v \notin V(L_{j_1})$. 
		Then we replace $P_1$ by the path $L_{j_1}^{-1}y$. If $x \not\sim w$ and $y \not\sim w$, then there exist $j_1$ and $j_2$ such that $\{y, z_1, \ldots, z_{n-2}\}  \cap V(L_{j_1}) = \emptyset$ and $\{x, z_1, \ldots, z_{n-2}\} \cap V(Q_{j_2}) = \emptyset$. In this case we replace $P_1$ by $L_{j_1}^{-1}Q_{j_2}$.
		
		{\bf Case 2.} There exist $i_0$ and $j_0$ such that $v$ belong to $V(L_{i_0})$ and $V(Q_{j_0})$.
		
		Since $v \in V(L_{i_0})$ and $w \nsim v$, there exists $t_1 \in \{y, z_1, \ldots, z_{n-2}\}$ such that $t_1 \in V(L_{i_0})$.
		Then $v$ and $t_1$ do not belong to $V(L_l)$ for any $1 \leq l \leq n, l \neq i_0$. There exists $i_1$ such that $V(L_{i_1}) \cap \{v, y, z_1, \ldots, z_{n-2}\} = \emptyset$. By similar argument there exists $j_1$ such that 
		$V(Q_{j_1}) \cap \{v, x, z_1, \ldots, z_{n-2}\} = \emptyset$. Then, we replace the path $P_1$ by the path 
		$L_{i_1}^{-1} Q_{j_1}$, which is internally disjoint from $P_1, \ldots, P_n$.
		
		{\bf Case 3.}  There exists $i_0$ such that $v \in V(L_{i_0})$ and $v \notin V(Q_j)$ for all $j$.
		
		By similar argument as of Case 2, there exists $i_1$ such that $V(L_{i_1}) \cap \{v, y, z_1, \ldots, z_{n-2}\} = \emptyset$. Since $v \notin V(Q_j)$ for all $j$ and $|\{x, z_1,\ldots, z_{n-2} \}| = n-1$, there exists $j_1$ such that 
		$Q_{j_1}$ is disjoint from $v, x, z_1, \ldots, z_{n-2}$.  We replace the path $P_1$ by $L_{i_1}^{-1} Q_{j_1}$
		and get $n$ internally disjoint $xy$-paths.
		
	\end{proof}

	\begin{lem} \label{lem:neighbor}
		Let $n \geq 1$ and let $G$ be an $n$-connected non-complete chordal graph.  Let $v$ be a simplicial vertex such that $G - \{v\}$ is $n$-connected. Then for any  $m \leq n$ and $\{x_1, \ldots, x_m\} \subseteq N_G(v)$, there exists $v' \neq v$  such that   $\{x_1, \ldots, x_m \}\subseteq N_G(v')$. 
	\end{lem}

	\begin{proof}
		Let $\{x_1, \ldots, x_m\} \subseteq N_G(v), m \leq n$. Since $G$ is $n$-connected, $deg(v) \geq n$.  If $deg(v) \geq n+1$, then clearly there exists a vertex $v' \in N_G(v) \setminus \{x_1, \ldots, x_m\}$. Since $v$ is simplicial $v' \sim x_i$ for all $1 \leq i \leq m$. So assume $deg(v) = n$.
	If $G - \{v\}$ is a complete graph, then since $G$ is non-complete, there exists $w \in V(G)$ such that 
		$w \neq v$ and $w \notin N_G(v)$.  In this case, we take $v' = w$. Assume that  $G-\{v\}$ is non-complete. Let $T$ be a maximal clique of $G -\{v\}$ containing $\{x_1, \ldots, x_m\}$. Since $G -\{v\}$ is non-complete and $n$-connected,  using  Proposition \ref{prop:cliquesize}, we conclude that $T$ is of size greater than $n$ and result follows.

	\vspace{-0.47cm}	
		
	\end{proof}

		\begin{thm} \label{thm:chordal} Let $n \geq 0$ and let $G$ be an $(n+1)$-connected  chordal graph. If $G$ cannot be folded onto a clique of size $n+2$, then  $\N(G)$ is  $n$-connected. 
	\end{thm}
	
	\begin{proof}
		
		Since $G$ is $(n+1)$-connected and chordal, it has a clique of size at least $n+2$. Suppose each maximal clique of $G$ has size $n+2$ by \Cref{prop:cliquesize}. Since $G$ is not folded onto a clique of size $n+2$, $G$ has at least two maximal cliques. Let $V_1$ be a maximal clique of $G$.   Then  from  \Cref{thm:simplicaildecomposition} and \Cref{prop:cliquesize}, the maximal cliques of $G$ can be arranged in a sequence $(V_1, \ldots, V_k)$ such that   $V_j \cap (\bigcup\limits_{i=1}^{j-1} V_i)$ is a clique of size $n+1$ for $2 \leq j \leq k$. Since $V_k$ is a clique of size $n+2$, we see that  there exists a vertex $v_k \in V_k$ such that $v_k \notin \bigcup\limits_{i=1}^{k-1} V_i$. Further, since $V_k \cap (\bigcup\limits_{i=1}^{k-1} V_i)$ is a clique of size $n+1$, we see that $G$ is folded onto $G - \{v_k\}$. Clearly, $G-\{v_k\}$ has simplicial decomposition $(V_1, \ldots, V_{k-1})$. From \Cref{prop:cliquesize}, we observe that $G - \{v_k\}$ is also $(n+1)$-connected. Since $G$ is not folded onto a clique of size $n+2$, $G-\{v_k\}$ is not a complete graph. Now, by similar argument there exists a $v_{k-1} \in V_{k-1}$ such that 
		$G- \{v_k, v_{k-1}\}$ is an $(n+1)$-connected non-complete  chordal graph. Since $k$ is finite, after $k$-steps, $G$ is folded onto $V_1$, which is a contradiction. Thus, $G$ has a clique of size at least $n+3$.

		If $n = 0$, then since $G$ has a clique of size $3$, $\chi(G) \geq 3$. It is well known that for any connected graph $G$ of chromatic number greater than $2$, $\N(G)$  is path connected. Proof is by induction on the number of vertices of the graph $G$. If $G$ is isomorphic to a complete graph $K_p$, then $p \geq n+3$ and in this case $\N(G) \simeq S^{p-2}$ and result is true. So, we assume that $G$ is non-complete. By Lemma \ref{lem:simplicial}, there exists  a simplicial vertex $v$ such that $G-\{v\}$ is $(n+1)$-connected and $\chi(G) = \chi(G-\{v\})$. Since $G$ has a clique of size $n+3$ and $G-\{v\}$ is perfect graph,  we see that 
		$G-\{v\}$ also has a clique of size $n+3$ and therefore $G-\{v\}$ cannot be folded onto  a clique of size $n+2$. 
		By induction hypothesis, $\N(G-\{v\})$ is $n$-connected. By Lemma \ref{lem:neighbor}, for any $S \subseteq N_G(v)$ such that, $|S| \leq n+1$ there exists a vertex $v_S \neq v$ such that $S \subseteq N_G(v_S)$. The result follows from Theorem \ref{thm:general}.
	\end{proof}

	The following is an immediate corollary of Theorem \ref{thm:chordal}.

	\begin{cor} Let $G$ be an $n$-connected chordal graph. If $\conn(\N(G)) < n-1$, then $\chi(G) = n+1$.
	\end{cor}
	\begin{proof}
	Using  \Cref{prop:cliquesize}, we conclude that  $\chi(G) \geq n+1$. Suppose $\chi(G) \geq n+2$. If $G$ is folded onto $K_{n+2}$, then from \Cref{fold}, $\N(G) \simeq \N(K_{n+2}) \simeq S^{n}$ and therefore $\conn(\N(G)) = n-1$. If $G $ is not folded onto $K_{n+2}$, then by Theorem \ref{thm:chordal} $\conn(\N(G)) \geq n-1$, which is a contradiction.
	\end{proof}

	\begin{thm}\cite[Theorem 9.19]{bondy} \label{thm:cut}
		Let $G$ be a connected chordal graph which is not complete, and let $S$ be a minimal vertex cut of $G$. Then $G[S]$ is a clique  of $G$. 
	\end{thm}

	We now as a consequence of Theorem \ref{thm:cutiscomplete} prove the following.

	\begin{thm}\label{thm:converse}
		Let $G$ be a chordal graph. If $G$ is stiff and vertex connectivity of $G$ is $n$,  then $\conn(\N(G)) < n$.
		
	\end{thm}
	\begin{proof}
		If $\chi(G) \leq n+2$, then by \Cref{lovasz}, $\conn(\N(G)) < n$. Let $\chi(G) \geq n+3$. Since $G$ is $n$-connected, from Theorem \ref{thm:chordal}, $\N(G)$ is $(n-1)$-connected. Let $S = \{x_1,  \ldots, x_n\}$ be a minimal vertex cut of $G$. From \Cref{thm:cut}, $G[S] \cong K_n$. Let $X_1, X_2, \ldots, X_r$ be the  $S$-components of $G$. Clearly, for each $1 \leq i \leq r, X_i$ is a $n$-connected chordal graph.

	Let $G_1= X_1$ and $G_2 = \cup_{i=2}^r X_i$. Observe that $G_1$ and $G_2$ are at least $n$-connected. 	From Proposition \ref{prop:cliquesize}, we conclude that there exist $a_i \in V(G_i)$ such that
		$a_i \sim x_j$ for all $1 \leq i \leq 2$ and $ 1 \leq j \leq n$.

		Since $G$ contain a clique of size $n+3$, either $G_1$ or $G_2$ contain a clique of size $n+3$. Suppose,  $G_1$ contains a clique of size $n+3$. From \Cref{thm:chordal}, $\N(G_1)$ is $(n-1)$-connected. If  all the maximal cliques of $G_2$ are of size $n+1$, then using \Cref{prop:cliquesize}, we conclude that $G_2$ is folded onto $K_{n+1}$.  Hence, $G_2$ can be folded onto $G_1$, which contradicts the fact that $G$ is stiff. 
		From Theorem \ref{thm:chordal}, we have thus  $\N(G_2)$ is  $(n-1)$-connected.  Using  Theorem \ref{thm:cutiscomplete}, we see that $\conn(\N(G)) = n-1$.
		By similar argument, if $G_2$ contains a  clique of size $n+3$, then  we can show that 
		$\conn(\N(G)) = n-1$.

	\end{proof}
	
	Combining Theorems \ref{thm:chordal} and \ref{thm:converse} we get our main result of this section.
	\begin{thm} \label{thm:chordalcomplete} Let $G$ be a non-complete stiff chordal graph. Then the vertex connectivity  $\kappa(G)= n+1$ if and only if $\conn(\N(G)) = n$.
	\end{thm}
	\begin{proof}
		If $\kappa(G) = n+1$, then Theorems \ref{thm:chordal} and \ref{thm:converse} imply that $\conn(\N(G)) = n$. Now, let  $\conn(\N(G)) = n$.  From  \Cref{thm:converse}, $\kappa(G) \geq n+1$. But, if $\kappa(G) \geq n+2$, then \Cref{thm:chordal} implies that $\N(G)$ is $(n+1)$-connected, which is a contradiction. Thus $\kappa(G)=n+1$. 
	\end{proof}
	
	%Combining Theorem \ref{thm:chordal} and Theorem \ref{thm:converse}, we have the following.
	
	%\begin{thm} Let $G$ be a connected stiff chordal  graph  having a clique of size $n+3$. Then $conn(\N(G)) =n$  if and only $\kappa(G) = n+1$?
	%	\end{thm}
		
	\begin{rmk}
	In \cite{csorba2}, Csorba proved that the box complexes of chordal graphs are homotopy equivalent to wedge of spheres. It is well known that box complex and neighbourhood complex are homotopy equivalent \cite{csorba3}. So, the neighbourhood complexes of chordal graphs are homotopy equivalent to wedge of spheres. 
	In his proof, Csorba used the simplicial decomposition of chordal graphs (see \Cref{thm:simplicaildecomposition}). He also remarked that using the simplicial decomposition structure, one can tell the possible dimensions of spheres appearing in the wedge. In fact, by following his proof and by using  \Cref{prop:cliquesize}, we can conclude that if $G$ is a $(n+1)$-connected, non-complete, stiff  chordal graph, then $\conn(\N(G)) \geq n$.
	\end{rmk}

%	\begin{rmk}
	
	\section{Concluding Remarks}\label{sec:conclusion}
	
In the previous section, we showed that the vertex connectivity of non-complete stiff chordal graphs is exactly one more than the connectivity of its neighbourhood complex. In the case of queen graphs, the vertex connectivity of the graph can be much larger than the connectivity of its neighbourhood complex, as shown in the following table.  	 
		\begin{table}[H]
	\centering
%	\footnotesize
\scalebox{0.58}{
	\begin{tabular}{|c|c|c|c|c|c|c|c|c|c|c|c|c|c|c|c|c|c|c|c|c|c|}
		\hline 
	(m, n)&$(2,2)$& $(2,3)$& $(2,4)$& $(2,5)$& $(2,6)$& $(2,7)$& $(2,8)$&$(2,9)$& $(2,10)$&$(3,3)$ & $(3,4)$& $(3,5)$& $(3,6)$ & $(3,7)$&$(3,8)$& $(4,2)$& $(4,4)$& $(4,5)$& $(4,6)$\\ [10pt]
	\hline
		$\kappa(Q_{m, n})$&$3$& $4$& $5$& $6$& $7$& $8$& $9$&$10$& $11$&$6$ & $7$& $8$& $9$ & $10$&$11$& $5$&  $9$& $10$& $11$\\ [5pt]
	\hline
    conn($N(Q_{m,n}$)&$1$ &$1$&$1$ &$2$&$2$&$2$&$2$&$2$&$2$&$2$&$2$&$2$&$2$&$2$&$2$&$1$&$2$&$2$&$2$ \\
    \hline
\end{tabular}}
\caption{\small{Vertex connectivity vs Connectivity of Neighbourhood Complex of $Q_{m, n}$}}
\label{table:1}
\end{table}	
 From \Cref{thm:weaklytriangulated},  for a class of  weakly triangulated graphs, $\kappa(G) > \text{Conn}(\N(G))$.  For any graph $G$ of chromatic number less or equal than three, clearly  from Theorem \ref{lovasz}, $\kappa(G) > \text{Conn}(\N(G))$. We can create more classes of graphs with vertex connectivity much larger than the connectivity of its neighbourhood complex by using the notion of Mycielskian of a graph.
	The Mycielskian $\sM(G)$ of $G$ with $V(G) = \{v_1, \ldots, v_n\}$ is a graph with 
	$V(\sM(G)) = \{v_1, \ldots, v_n\} \sqcup \{u_1, \ldots, u_n\} \sqcup \{w\}$ and 
	$E(\sM(G))  = E(G) \sqcup \{(u_i, v_j), (u_j, v_i) | (v_i, v_j) \in E(G)\} \sqcup \{(w, u_j )| 1 \leq j \leq n\}. $

%	\begin{rmk}\label{vertexconnvsconn}
	 In \cite{csorba1}, Csorba proved that the neighbourhood complex $\N(\sM(G))$ is homotopy equivalent to the suspension of $\N(G)$. In \cite{chang}, Chang et al. shown that  $\kappa(\sM(G)) > \kappa(G)$. Hence,  using the results of Csorba and Chang et al., we have  $\kappa(G) > \text{Conn}(\N(G))$ implies that $\kappa(\sM(G)) > \text{Conn}(\N(\sM(G)))$. From 
	 \Cref{thm:chordalcomplete}, for non-complete stiff chordal graph $G$, $\kappa(G) > \text{Conn}(\N(G))$. So, for all of these classes of graphs, vertex connectivity is larger than the connectivity of its neighbourhood complex.

     %Now, we construct examples where, vertex connectivity is smaller than the connectivity of the neighbourhood complex of a graph. 
     For $n \geq 5$,  let $\tilde{K}_n$ be the graph defined as a complete graph $K_n$ with one hanging edge (see \Cref{Fig:a}), then clearly $\tilde{K}_n$ is folded onto $K_n$ and its vertex connectivity is $1$. But $\N(\tilde{K}_n) \simeq \N(K_n) \simeq S^{n-2}$ and therefore Conn$(\N(\tilde{K}_n)) = n-3$. More generally, we construct  a class of stiff graphs, where the vertex connectivity is arbitrarily smaller than the topological connectivity of its neighbourhood complex.

\begin{figure} [H]
	\begin{subfigure}[]{0.4\textwidth}
				\vspace{0.5 cm}
		\centering
		\vspace{0.2 cm}
		\begin{tikzpicture}
		[scale=0.35, vertices/.style={draw, fill=black, circle, inner sep=0.95pt, minimum size = 0pt, }]
		
		\node[vertices, label=left:{}] (1) at (0,0) {};
		\node[vertices, label=below:{}] (2) at (4,0) {};
		\node[vertices, label=below:{}] (3) at (0, 3.5) {};
		\node[vertices, label=above:{}] (4) at (4, 3.5) {};
		
		\node[vertices, label=left:{}] (5) at (2, 5.5) {};
			\node[vertices, label=left:{}] (6) at (4, 5.5) {};

		\foreach \to/\from in
		{1/2, 1/3, 1/4, 1/5, 2/3, 2/4, 2/5, 3/4, 3/5, 4/5, 5/6} \draw [-] (\to)--(\from);
		%	\foreach \to/\from in
		%	{00/10, 10/20, 20/30, 30/40, 00/01, 01/11, 12/32, 11/21, 21/31, 31/41, 40/41, 01/02,02/03, 10/11, 20/21, 30/31, 02/12, 32/42, 11/12, 21/22, 31/32, 41/42, 12/13, 22/23, 32/33, 42/43, 03/13, 13/23, 23/33, 33/43} \draw [dashed] (\to)--(\from);
		%	

		\end{tikzpicture}
				\vspace{0.2 cm}
		\caption{$\tilde{K}_5$} \label{Fig:a}
	\end{subfigure}
	\begin{subfigure}[]{0.6\textwidth}
		\centering
		\begin{tikzpicture}
		[scale=0.4, vertices/.style={draw, fill=black, circle, inner
			sep=0.95pt}]
		\node[vertices, label=below:{9}] (9) at (0,0) {};
		\node[vertices, label=below:{10}] (10) at (2, 0) {};
		\node[vertices,label=right:{1}] (1) at (3.5,1) {};
		\node[vertices,label=right:{2}] (2) at (4.2,2.8) {};
		\node[vertices,label=right:{3}] (3) at (3.5,4.5) {};
		\node[vertices,label=above:{4}] (4) at (2,5.5) {};
		\node[vertices,label=above:{5}] (5) at (0,5.5) {};
		\node[vertices,label=left:{6}] (6) at (-1.5,4.5) {};
			\node[vertices,label=left:{7}] (7) at (-2.2,2.8) {};
				\node[vertices,label=left:{8}] (8) at (-1.5,1) {};
				
					\node[vertices,label=below:{11}] (11) at (6,2.8) {};
						\node[vertices,label=below:{12}] (12) at (8,1.2) {};
				\node[vertices,label=below:{13}] (13) at (10.5,1.2) {};
				\node[vertices,label=above:{14}] (14) at (10.5,4.3) {};
					\node[vertices,label=above:{15}] (15) at (8,4.3) {};

		\foreach \to/\from in
		{9/10, 10/1, 9/8, 1/2, 7/8, 2/3, 6/7, 3/4, 4/5, 5/6, 11/12, 11/13, 11/14, 11/15, 12/13, 12/14, 12/15, 13/14, 13/15, 14/15, 11/1, 11/3, 1/4, 4/7, 7/10, 10/1, 2/5, 5/8, 8/1, 10/3, 3/6, 6/9, 9/2} \draw [-] (\to)--(\from);
		
		\end{tikzpicture}
		\caption{$G_{1, 5}$} \label{Figure:b}
	\end{subfigure}

	\caption{Examples of the graphs $\tilde{K}_n$ and $G_{r, p}$} \label{ex}
\end{figure}

	Let $n \geq 2$ be a positive integer and $S \subset \{ 1, 2, \ldots, n-1\}$. 
	The {\it circulant graph} $C_n(S)$ is the graph, whose set of vertices
	$V(C_n(S)) =  \{1, 2, \ldots, n\}$  and any two vertices $x$ and $y$
	adjacent if and only if $x-y \ (\text{mod \ n}) \in S \cup -S$, where $-S = \{n-a  :  a \in S\}$. 	Observe  that 
for any $t \in \{1, 2, \ldots, n\}$, $N_{C_n(S)}(t) = \{s+t : s \in S\} \cup \{ n-s+t : s \in S \}$.
	
Let $r \geq 1, p \geq 3, n = 4r+6$ and $S = \{1, 3, \ldots, 2r+ 1\}$. Let $H_n^p$ be the  complete graph on vertex set 
$\{n+1, \ldots, n+p\}$. Let $G_{r,p}$ be the  graph  on vertex set $\{1, 2, \ldots, n+p\}$ and 
	$E(G_{r,p}) = E(C_n(S)) \cup E(H_n^p) \cup \{( n+1, 1), (n+1, 3)\}$ (see \Cref{Figure:b}).  For a set $A$, let $\Delta^{A}$ denote a simplex on vertex set $A$ and $\partial(\Delta^{A})$, the simplicial boundary of $\Delta^{A}$.  It is easy to check that 
$\N(C_n(S)) = \partial(\Delta^{\{1, 3, 5, \ldots, n-1\}}) \sqcup \partial(\Delta^{\{2, 4,  6, \ldots, n\}})$. Hence $\N(C_n(S)) \simeq S^{2r+1} \sqcup S^{2r+1} $. Since $\N(H_n^p) \simeq S^{p-2}$,  by \Cref{claim:example}, we conclude that $\N(G) \simeq S^{2r+1} \vee S^{2r+1} \vee S^{p-2}$. By construction, the vertex connectivity of $G_{r,p}$ is $1$ whereas the connectivity of its neighbourhood complex can be made very large by making appropriate choices of $r$ and $p$.

	\subsection*{Acknowledgement}
	
	The second author would like to thank Niranjan Balachandran for his insights about chordal graphs. 

	\appendix
		\section{Appendix}

	\begin{thm}\cite[Theorem 9.20]{bondy} \label{thm:simplicaildecomposition} Let $G$ be a  chordal graph and let $V_1$ be a maximal clique of $G$. Then   the maximal cliques of $G$ can be arranged in a sequence  $(V_1, \ldots, V_k)$ such that   $V_j \cap (\bigcup\limits_{i=1}^{j-1} V_i)$ is a clique of $2 \leq j \leq k$.  Such a sequence $(V_1, \ldots, V_k)$ is called a simplicial decomposition of $G$.
	\end{thm}
	
	The following statement is probably well known to experts. We are proving it here for completeness. 
	\begin{prop} \label{prop:cliquesize}Let $G$ be an $n$-connected chordal graph  and let $V_1$ be a maximal clique of $G$. Then either $G = V_1$  or the maximal cliques of $G$ can be arranged in a sequence  $(V_1, \ldots, V_k)$ such that   $V_j \cap (\bigcup\limits_{i=1}^{j-1} V_i)$ is a clique of size at least $n $, $2 \leq j \leq k$. 
	\end{prop}
	
	\begin{proof} Assume that $G \neq V_1$, {\it i.e.,} $G$ is non-complete. Proof is by induction on number of vertices of $G$. Let $S$ be a minimal vertex cut of $G$ and let $X_1, \ldots, X_N$ are the $S$-components of $G$. Since $G$ is $n$-connected, $|S| \geq n$. Clearly, each $X_i$ is an $n$-connected chordal graph. Without loss of generality we assume that $V_1$ is a maximal clique of $X_1$. From \Cref{thm:cut}, $G[S]$ is a clique.  Let $C^{i}$ be a maximal clique of $X_i$ containing $S$, $1 \leq i \leq N$. Clearly, each maximal clique of $X_i$ is also a maximal clique of $G$. By induction, for each $1 \leq i \leq N$,  either $X_i$ is a complete graph or we can arrange the maximal cliques of $X_i$ by $ M_i = (T_{j_1}^{i}, \ldots, T^{i}_{J_{l_i}})$ such that $T^{i}_{j_t} \cap (\bigcup\limits_{m=1}^{t-1} T^i_{j_m}), 2 \leq t \leq l_i $ is a clique of size at least $n$, where $T^{1}_{j_1} = V_1$ and $T^{i}_{j_1} = C^{i}$ for $2 \leq i \leq N$. 
			Let  $M = (T_1, \ldots, T_{l_1}, T_{l_1+1}, \ldots, T_{l_1+l_2}, \ldots, T_{l_1 + \cdots + l_{N-1}+1}, \ldots,  T_{l_1 + \cdots +l_N})$, where $M_1= (T_1, \ldots, T_{l_1})$ and $M_i = (T_{l_1 + \cdots+ l_{i-1}+1}, \ldots, T_{l_1+ \cdots+ l_{i}})$ for $2 \leq i \leq N$.  Let $2 \leq t \leq l_1 + \ldots + l_N$. If $t \leq l_1$, then since $M_1$ is a simplicial decomposition of $X_1$, $ T_t \cap (\bigcup\limits_{i=1}^{t-1} T_i)$ will be a clique of size at least $n$. 
		So assume $t \geq l_1+1$. There exists $1 \leq p \leq N-1$, such that $l_1 + \ldots  + l_{p}+1 \leq t\leq  l_1 + \ldots + l_{p}+ l_{p+1} $. If $t= l_1 + \ldots + l_{p}+1$, then  $V(T_t \cap (\bigcup\limits_{i=1}^{t-1} T_i)) = S$ and therefore $T_t \cap (\bigcup\limits_{i=1}^{t-1} T_i)$ is a clique of size at least $n$ by Theorem \ref{thm:cut}. Let $t > l_1 + \ldots + l_{p}+1$. Observe that the vertices of $T_t \cap (\bigcup\limits_{i=1}^{t-1} T_i)$ is a subset of vertices of $X_{p+1}$. Hence, $ T_t \cap (\bigcup\limits_{i=1}^{t-1} T_i) =  T_t \cap (\bigcup\limits_{i = l_1 + \ldots  + l_{p}+1}^{t-1} T_i) $. Since $M_{p+1}$ is a simplicial decomposition of $X_{p+1}$, by induction $ T_t \cap (\bigcup\limits_{i = l_1 + \ldots  + l_{p}+1}^{t-1} T_i)$ is a clique of size at least $n$. Hence, $M$ is a simplicial decomposition of $G$ such that $T_t \cap (\bigcup\limits_{i=1}^{t-1} T_i)$ is a clique of size at least $n$, $2 \leq t \leq l_1 + \ldots + l_N$.
	\end{proof}
	
			Let  $X$ be a simplicial complex and $\tau, \sigma \in X$ such that
	$\sigma \subsetneq \tau$ and  $\tau$ is the only maximal simplex in $X$ that contains $\sigma$.
	A  {\it simplicial collapse} of $X$ is the simplicial complex $Y$ obtained from $X$ by
	removing all those simplices $\gamma$  of $X$ such that
	$\sigma \subseteq \gamma \subseteq \tau$. Here, $\sigma$ is called a {\it free face} of
	$\tau$ and $(\sigma, \tau)$ is called a {\it collapsible pair}. It is well known that, if 
	$X$ collapses to $Y$, then $X \simeq Y$ \cite[Proposition 6.14]{dk}.

	Recall that for a simplicial complex $\K$, the set of maximal simplices of $\K$ denoted by 
	$\M(\K)$.
	\begin{claim} \label{claim:example}
Let $r \geq 1, p \geq 3, n = 4r+6$ and $S = \{1, 3, \ldots, 2r+ 1\}$. Let  the graphs $G_{r, p}$ and $H_n^p$ be as defined in 	\Cref{sec:conclusion}. Then 	$\N(G_{r, p})$ collapses to a subcomplex $X$, where 
	$$ \M(X) = \M(\N(C_n(S))) \sqcup \M(\N(H_n^p)) \sqcup \{\{3, n+p\} \}\sqcup \{\{n, n+1\}\}.$$
	
		\end{claim}
	
	\begin{proof} For convenience of notation, we denote the graph  $G_{r, p}$ by $G$ and the graph $H_n^p$ by $H$. Clearly, 
	$N_G(1) = N_{C_n(S)}(1) \cup \{n+1\}, N_G(3) = N_{C_n(S)}(3) \cup \{n+1\}, N_G(n+1) = \{n+2, \ldots, n+p\} \cup \{1, 3\}$. For any $i \in [n] \setminus \{1, 3\}$, 
	$N_G(i) = N_{C_n(S)}(i)$ and for any $j \in \{n+2, \ldots, n+p\}$, $N_G(j) = N_H(j)$.
	
		Observe that  $(\{1, n+2\}, N_G(n+1))$ is a collapsible pair in $\N(G)$. By using this collapsible pair,  $N_G(n+1)$ collapses to 
		$\delta_1 = N_G(n+1) \setminus \{n+2\}$ and $\delta_2 = N_G(n+1) \setminus \{1\}$. Thus, $\N(G)$ collapses to a subcomplex $\Delta_1$, where  $\M(\Delta_1) = (\M(\N(G)) \setminus \{N_G(n+1)\}) \sqcup \{\delta_1, \delta_2\}$. Now $(\{1, n+3\}, \delta_1)$ is a collapsible pair in $\Delta_1$ and therefore $\delta_1$  collapses to $\delta_1 \setminus \{n+3\} $ and $\delta_1 \setminus \{1\} \subseteq \delta_2$.  Hence, $\Delta_1$ collapses to  $\Delta_2$, where $\M(\Delta_2) = (M(\N(G)) \setminus \{N_G(n+1)\}) \sqcup \{\delta_2, \delta_1 \setminus \{n+3\}\}$.  Since $\{1, 3\} \subseteq N_G(2)$, 	by applying a sequence 
		$$
		(\{1, n+4\}, \delta_1 \setminus \{n+3\} ), 	(\{1, n+5\}, \delta_1 \setminus \{n+3, n+4\}), \ldots, (\{1, n+p\}, \delta_1 \setminus \{n+3, \ldots, n+p-1\})
		$$
	of collapsible pairs, $\Delta_2$ collapses to  $\Delta_3$, where
	$\M(\Delta_3) =  (\M(\N(G)) \setminus \{N_G(n+1)\}) \sqcup \{\delta_2\}$.	
	
	Observe that for any $2 \leq i \leq p$, $(\{3, n+i\}, \delta_2)$ is a collapsible pair in $\Delta_3$. Hence by using the collapsible pair  $(\{3, n+2\}, \delta_2)$, $\delta_2 $ is collapses to $\delta_2 \setminus\{3\} = N_{H}(n+1)$ and $\delta_2 \setminus \{n+2\}$. Let $\sigma = \delta \setminus\{n+2\}$. Then,  $\Delta_3$  collapses to a subcomplex
	$\Delta_4$, where the set of maximal simplices 
	$\M(\Delta_4) =  (\M(\N(G)) \setminus \{N_G(n+1)\} )\sqcup \{N_H(n+1), \sigma\}$.	Now $(\{3, n+3\}, \sigma)$ is a collapsible pair in $\Delta_4$ and therefore $\sigma$ is collapses to 
	$\sigma \setminus \{n+3\}$ and $\sigma \setminus \{3\} \subset N_H(n+1)$.   	By applying a sequence 
	$$
	(\{3, n+4\}, \sigma \setminus \{n+3\} ), 	(\{3, n+5\}, \sigma \setminus \{n+3, n+4\}), \ldots, (\{3, n+p-1\}, \sigma \setminus \{n+3, \ldots, n+p-2\})
	$$
	of collapsible pairs,  we see that $\Delta_4$ collapses to a subcomplex $\Delta_5$, where the set of maximal simplices 
	$\M(\Delta_5) = ( \M(\N(G)) \setminus \{N_G(n+1)\} )\sqcup \{N_H(n+1)\} \sqcup \{\{3, n+p\}\}$.	
	
	Observe that $(\{n+1, 2r+6\}, N_G(1))$ and $(\{n+1, 2r+4\}, N_G(3))$ are  collapsible pairs in $\N(G)$ and hence in  $\Delta_5$. Therefore, by using these collapsible pairs we get that  $N_G(1)$ is collapses to $N_G(1) \setminus \{n+1\}$ and $N_G(1) \setminus \{2r+6\}$, and  $N_G(3)$ is collapses to $N_G(3) \setminus \{n+1\}$ and $N_G(3) \setminus \{2r+4\}$. Clearly,
	$N_G(1) \setminus \{n+1\} = N_{C_n(S)}(1)$, $N_G(3) \setminus \{n+1\} = N_{C_n(S)}(3)$ and 
	$N_G(1) \setminus \{2r+6\} = N_G(3) \setminus \{2r+4\} = \{n+1\} \cup \{2, 4, 6, \ldots, n\} \setminus \{2r+4, 2r+6\}$. Let $\tau = N_G(1) \setminus \{2r+6\}$. 
	Hence  $\Delta_5$ collapses to a subcomplex $\Delta_6$, where 
	\begin{align*}
	\M(\Delta_6)  = &  ( \M(\N(G)) \setminus \{N_G(n+1), N_G(1), N_G(3)\} ) \\
	& \sqcup \{N_H(n+1), \{3, n+p\}, N_{C_n(S)}(1), N_{C_n(S)}(3), \tau \}.
	\end{align*}
	Observe that  for any $j \in \tau, j \neq n+1$, $(\{n+1, j\}, \tau)$ is a collapsible pair in $\Delta_6$. Since $\tau \setminus \{n+1\} \subseteq N_{C_n(S)}(1)$,  	by applying a sequence 
	$$
	(\{n+1, 2\}, \tau), ( \{n+1, 4\}, \tau \setminus \{2\} ) \ldots (\{n+1, 2r+2\}, \tau\setminus \{2, 4,  \ldots, 2r\}), 
	$$
	$$
	(\{n+1, 2r+8\}, \tau\setminus \{2, 4,  \ldots, 2r, 2r+2\}), (\{n+1, 2r+10\}, \tau\setminus \{2, 4,  \ldots, 2r, 2r+2, 2r+8\}), 
	$$
	$$
	 \ldots,  (\{n+1, 4r+4\}, \sigma \setminus \{2, 4,  \ldots, 2r+2, 2r+8, \ldots, 4r+2 \})
	$$
	of collapsible pairs,  we see that $\Delta_6$ collapses to a subcomplex $\Delta_7$, where 
\begin{align*}
    \M(\Delta_7) =&  (\M(\N(G)) \setminus \{N_G(n+1), N_G(1), N_G(3)\} )\sqcup \{N_H(n+1),N_{C_n(S)}(1), N_{C_n(S)}(3)\} \\
    & \sqcup\{\{3, n+p\}, \{n+1, n\}\}.
\end{align*}
Clearly, $\M(\Delta_7) = \M(\N(C_n(S))) \sqcup \M(\N(H)) \sqcup \{\{3, n+p\} \}\sqcup \{\{n, n+1\}\}$. We  take $X= \Delta_7$.
		\end{proof}

	\bibliographystyle{alpha}

\end{document}